\documentclass[11pt]{article}
\usepackage{amssymb,amsmath,amsthm,graphicx,hyperref}
\usepackage{color,enumerate}
\usepackage{authblk}
\usepackage{cite,fullpage}
\usepackage{caption}
\usepackage{subcaption}
\usepackage[ruled,vlined,linesnumbered]{algorithm2e}
\usepackage{relsize}
\usepackage{rotating}
\usepackage[utf8]{inputenc}

\usepackage{tabularx}

\newtheorem{theorem}{Theorem}[section]
\newtheorem{claim}[theorem]{Claim}

\newtheorem{corollary}[theorem]{Corollary}
\newtheorem{proposition}[theorem]{Proposition}

\theoremstyle{remark}
\newtheorem*{remark}{Remark}

\theoremstyle{definition}
\newtheorem{definition}[theorem]{Definition}

\title{Mind the Independence Gap}

\author[1]{T\i naz Ekim\thanks{tinaz.ekim@boun.edu.tr}}
\author[4]{Didem Gözüpek\thanks{didem.gozupek@gtu.edu.tr}}
\author[2,3]{Ademir Hujdurovi\'c\thanks{ademir.hujdurovic@upr.si}}
\author[2,3]{Martin Milani\v{c}\thanks{martin.milanic@upr.si}}

\affil[1]{\normalsize Department of Industrial Engineering, Bo\u{g}azi\c{c}i University, Istanbul, Turkey}
\affil[2]{\normalsize University of Primorska, UP IAM, Muzejski trg 2, SI-6000 Koper, Slovenia}
\affil[3]{\normalsize University of Primorska, UP FAMNIT, Glagolja\v ska 8, SI-6000 Koper, Slovenia}
\affil[4]{\normalsize Department of Computer Engineering, Gebze Technical University, Kocaeli, Turkey}

\date{\today}

\usepackage{amssymb}
\usepackage{graphicx}
\begin{document}

\maketitle
\begin{abstract}
The independence gap of a graph was introduced by Ekim et al.~(2018) 
 as a measure of how far a graph is from being well-covered. It is defined as the difference between the maximum and minimum size of a maximal independent set.

We investigate the independence gap of a graph from structural and algorithmic points of view, with a focus on classes of perfect graphs. Generalizing results on well-covered graphs due to Dean and Zito (1994) and Hujdurovi\'c et al.~(2018), we express the independence gap of a perfect graph in terms of clique partitions and use this characterization to develop a polynomial-time algorithm for recognizing graphs of constant independence gap in any class of perfect graphs of bounded clique number. Next, we introduce a hereditary variant of the parameter, which we call \emph{hereditary independence gap} and which measures the maximum independence gap over all induced subgraphs of the graph. We show that determining whether a given graph has hereditary independence gap at most $k$ is polynomial-time solvable if $k$ is fixed and co-NP-complete if $k$ is part of input. We also investigate the complexity of the independent set problem in graph classes related to independence gap, showing that the problem is NP-complete in the class of graphs of independence gap at most one and polynomial-time solvable in any class of graphs with bounded hereditary independence gap. Combined with some known results on claw-free graphs, our results imply that the independent domination problem is solvable in polynomial time in the class of $\{$claw, $2P_3\}$-free graphs.
\end{abstract}

\section{Introduction}

An \emph{independent set} in a graph is a set of pairwise non-adjacent vertices. An independent set in a graph is said to be \emph{maximal} if it is not contained in any larger independent set. For a graph $G$ we denote with $\alpha(G)$ the maximum size of an independent set in $G$, called the \emph{independence number} of $G$, and with $i(G)$ its \emph{independent domination number}, that is, the minimum size of a maximal independent set in $G$. Computing the independence number in various graph classes has been a central problem for decades. Similarly, many research papers focused on the recognition of well-coveredness in graph classes;
a graph is said to be \emph{well-covered} if all its maximal independent sets have the same size~\cite{MR1254158,MR1677797}.
A measure of how far a graph $G$ is from being well-covered was introduced
by Ekim et al.~in~\cite{DBLP:journals/corr/abs-1708-04632} under the name \emph{independence gap} of $G$, defined as the difference $\alpha(G)-i(G)$ and denoted by $\mu_\alpha(G)$. Note that $\mu_\alpha(G)= 0$ if and only if $G$ is well-covered. Along this line, a graph is called \emph{almost well-covered} if $\mu_\alpha(G)= 1$.

In this paper, we investigate the independence gap of a graph following three main approaches. Our results and motivations can be summarized as follows.

\medskip
{\bf \noindent 1.~Independence gap of perfect graphs.} We study the problem of computing the independence gap under various conditions, obtaining both positive and negative results. As shown by~Gr{\"o}tschel et al.~\cite{grotschel1981ellipsoid}, the independent set problem can be solved in polynomial time in the class of perfect graphs. The problem of computing the independence gap is more difficult, since it relates also to the independent domination problem and the problem of recognizing well-covered graphs, both of which are known to be intractable for perfect graphs. Indeed, these connections along with known results in the literature imply that computing the independence gap is NP-hard for bipartite graphs and for weakly chordal graphs (in the case of weakly chordal graphs even for any constant value of the parameter). As a generalization of perfect graphs, a graph is said to be \emph{semi-perfect} if its vertex set can be covered with $\alpha(G)$ cliques. We express the independence gap of semi-perfect graphs as the smallest value of $k$ such that the vertex set of the graph can be covered with a family of pairwise disjoint cliques such that each maximal independent set intersects all but at most $k$ cliques in the partition (Theorem~\ref{thm:semiperfect} and Corollary~\ref{cor:semiperfect-value-of-the-gap}). This result generalizes a characterization of well-covered semi-perfect graphs due to Hujdurovi\'c et al.~\cite{localizable} (which corresponds to the case $k = 0$). The result leads to a polynomial-time recognition algorithm of bipartite graphs of constant independence gap and, more generally, of graphs of constant independence gap in any class of perfect graphs of bounded clique number (Corollary~\ref{cor:perfect}). This algorithmic result generalizes the fact that well-covered graphs can be recognized in polynomial time in any class of perfect graphs of bounded clique number, as shown by Dean and Zito~\cite{MR1264476}.

\medskip
{\bf \noindent 2.~The independent set problem in graphs of small independence gap.} Clearly, in any class of graphs of constant independence gap, a constant additive approximation to a maximum independent set can be obtained simply by computing and returning any maximal independent set. This makes interesting the question about the complexity of the independent set problem in graphs of independence gap at most $k$, where $k$ is a positive integer. We show that in the class of graphs of independence gap at most one, the independent set problem is NP-complete and recognizing if a graph is well-covered is co-NP-complete (Theorem~\ref{thm:hardness}).
In particular, this means that even if we know that all maximal independent sets of a graph have size either $\alpha(G)$ or $\alpha(G)-1$, it is still hard to compute the exact value of $\alpha(G)$.

\medskip
{\bf \noindent 3.~A hereditary variant of independence gap.} It is not difficult to see that deleting a vertex from a graph may change the value of its independence gap in either direction and that classes of graphs of bounded independence gap are not hereditary.\footnote{A class of graphs is \emph{hereditary} if it is closed under vertex deletions.} We introduce a ``hereditary'' version of the parameter, called \emph{hereditary independence gap}, which measures the maximum independence gap over all induced subgraphs of the graph. We show that for every constant $k$, the class of graphs of hereditary independence gap at most $k$ is characterized by a finite set of forbidden induced subgraphs (Theorem \ref{thm:hereditary-k}). This result has several consequences, which indicate a variety of aspects in which the hereditary version of the parameter differs from the usual one.
First, for every $k$, graphs of hereditary independence gap at most $k$ can be recognized in polynomial time (Corollary \ref{cor:hereditarily-k-quasi-well-covered}). Second, combined with a result of Lozin and Rautenbach~\cite{WIS}, our approach shows that for every fixed $k$, the weighted independent set problem is solvable in polynomial time in the class of hereditary independence gap at most $k$ (Corollary~\ref{cor:WIS}). Third, in the special case of $k = 1$, our characterization (Theorem \ref{thm:hereditary}) combined with known results on claw-free graphs leads to a new polynomially solvable case of the independent domination problem -- the class of $\{$claw, $2P_3$$\}$-free graphs (Corollary \ref{cor:inddom}).
Finally, we complement the result about polynomial-time recognition of graphs of constant hereditary independence gap by showing that the computation of the hereditary independence gap of a given graph is NP-hard (Theorem~\ref{thm-np-hardness}).

\medskip
\noindent{\bf Overview of related work.}
Graphs of zero independence gap are the well-covered graphs, which are well studied in the literature, see, e.g., the survey papers~\cite{MR1254158,MR1677797}. In particular, the problem of determining whether a given graph is well-covered is known to be co-NP-complete~\cite{MR1161178,MR1217991}. The class of almost well-covered graphs, that is,  graphs of independence gap one, was denoted by $I_2$ in the paper by Barbosa and Hartnell~\cite{MR1676478}. They characterized almost well-covered simplicial graphs and gave a sufficient condition for a chordal graph to be almost well-covered. Ekim et al.~\cite{DBLP:journals/corr/abs-1708-04632} investigated almost well-covered graphs of girth at least six.

Clearly, every graph with independence gap at most $k$ has the property that its maximal independent sets are of at most $k+1$ different sizes. Finbow, Hartnell, and Whitehead denoted in~\cite{awc_girth8} by $M_k$ the class of graphs that have maximal independent sets of exactly $k$ different sizes. These graphs were studied further by Hartnell and Rall~\cite{MR3053598} and by Barbosa et al.~\cite{MR3056978}. Since every graph with independence gap at most $k$ is in $M_r$ for some $r\in \{1,\ldots, k+1\}$, some results on graphs in classes $M_r$ have implications for graphs with independence gap at most $k$. For example, a result due to Barbosa et al.~\cite[Theorem 2]{MR3056978} implies that for every $k$ and $d$, there are only finitely many connected graphs with independence gap at most $k$, minimum degree at least $2$, maximum degree at most $d$, and girth at least~$7$.

\medskip
\noindent{\bf Structure of the paper.} Each of Sections~\ref{sec:indgap},~\ref{sec:computing-alpha},~and~\ref{sec:hereditary} is devoted to one of the three main themes outlined above. We conclude the paper in Section \ref{sec:sum} with a table summarizing the results of this paper along with related results from the literature and some open questions.

\subsection{Preliminaries}

All graphs considered in this paper are finite, simple, and undirected.
A \emph{clique} in a graph is a set of pairwise adjacent vertices. A clique is \emph{maximal} if it is not contained in any larger clique. The \emph{clique number} of a graph $G$, denoted by $\omega(G)$, is the maximum size of a clique in $G$. For a graph $G$, its complement is denoted by $\overline G$. As usual, we denote the $n$-vertex path and complete graph by $P_n$ and $K_n$, respectively. The complete bipartite graph with parts of sizes $m$ and $n$ is denoted by $K_{m,n}$.
Given two vertex sets $A$ and $B$ in a graph $G$, we say that $A$ \emph{dominates} $B$ if every vertex in $B$ has a neighbor in $A$. In particular, a set $S$ of vertices in a graph $G$ is a \emph{dominating set} if $S$ dominates $V(G)\setminus S$. The {\sc Independent Set} problem takes as input a graph $G$ and an integer $k$ and the task is to determine whether $\alpha(G)\ge k$. Similarly, given a graph $G$ and an integer $k$, the problem of deciding whether $i(G)\le k$ is the {\sc Independent Domination} problem. Note that an independent set $I$ in a graph $G$ is a dominating set if and only if $I$ is a maximal independent set. The {\sc Weighted Independent Set} problem takes as input a vertex-weighted graph and the task is to compute an independent set of maximum total weight.

Given a set of graphs $\mathcal{F}$, a graph $G$ is said to be \emph{$\mathcal{F}$-free} if it contains no induced subgraph isomorphic to a member of $\mathcal{F}$. If $\mathcal{F} = \{H\}$ for a graph $H$, we write $H$-free instead of $\{H\}$-free. A class of graphs is \emph{hereditary} if it is closed under vertex deletions. It is well known that every hereditary class of graphs can be characterized by a set of \emph{forbidden induced subgraphs}, that is, there exists a unique minimal set $\mathcal{F}$ of graphs such that a graph $G$ is in the class if and only if $G$ is $\mathcal{F}$-free.

\section{Independence gap of perfect graphs}\label{sec:indgap}

We start with some preliminary observations about the computational complexity of the problem of computing the independence gap. First, let us notice that recognizing graphs of any constant independence gap is co-NP-complete.

\begin{proposition}\label{prop:all}
For every fixed integer $k\ge 0$, it is co-NP-complete to determine if a given graph $G$ satisfies $\mu_{\alpha}(G)\le k$.
\end{proposition}

\begin{proof}
Given a no instance $G$ to the problem, a short certificate for the fact that $\mu_{\alpha}(G)> k$ consists of two sets $(I_1, I_2)$ of vertices such that $I_1$ and $I_2$ are maximal independent sets and $||I_1|-|I_2||\ge k+1$.
Since these conditions for $I_1$ and $I_2$ can be verified in polynomial time, the problem is in co-NP.

We prove NP-hardness using a reduction from the co-NP-complete problem of recognizing well-covered graphs~\cite{MR1161178,MR1217991}. Let $G$ be the input weakly chordal graph for the recognition of well-coveredness and let $G'$ be the disjoint union of $G$ and the complete bipartite graph $K_{1,k}$. Then, it can be observed that $G$ is well-covered if and only if $G'$ has independence gap at most $k$.
\end{proof}

A graph $G$ is said to be \emph{weakly chordal} if neither $G$ nor its complement contain an induced cycle of length at least $5$. The reduction showing co-NP-completeness of the problem of recognizing well-covered graphs~from~\cite{MR1161178,MR1217991} actually shows that the problem is co-NP-complete even when restricted to the class of weakly chordal graphs.
Since the reduction in the proof of Proposition~\ref{prop:all} maps a weakly chordal graph $G$ to a weakly chordal graph $G'$, we thus obtain the following.

\begin{corollary}\label{cor:w-c}
For every fixed integer $k\ge 0$, it is co-NP-complete to determine if a given weakly chordal graph $G$ satisfies $\mu_{\alpha}(G)\le k$.
\end{corollary}

\begin{corollary}\label{cor:complexity}
The problem of computing the independence gap of a given graph is NP-hard. \end{corollary}

In view of Corollaries~\ref{cor:w-c} and~\ref{cor:complexity}, it is interesting to identify restrictions on the input graphs under which the independence gap can be computed in polynomial time, or, when this is not possible (unless P = NP), whether one can at least efficiently recognize graphs in the class that are of constant independence gap.

Clearly, the independence gap is computable in polynomial time in any class of graphs in which both the independence number and the independent domination number are polynomially computable. This includes the classes of chordal graphs \cite{gavril1972algorithms,MR687354}, circular arc graphs \cite{MR1143909,MR1622646}, permutation graphs~\cite{MR800722} and, more generally, cocomparability graphs~\cite{kohler2016linear,MR1229694} and AT-free graphs \cite{broersma1999independent}, and graphs of bounded clique-width~\cite{MR1739644}. On the other hand, the problem of computing the independence gap is NP-hard in any class of graphs in which the independence number can be computed in polynomial time but computing the independent domination number is NP-hard. This is the case for example for the classes of line graphs~\cite{MR579424}, bipartite graphs~\cite{MR754426}, and weakly chordal graphs~\cite{MR2207507} (which also follows from Corollary~\ref{cor:w-c}).\footnote{The problem of computing the independence gap is also NP-hard in any class of graphs in which the independent domination number can be computed in polynomial time but computing the independence number is NP-hard. However, we are not aware of any natural graph class with this property.}

Note that bipartite graphs are perfect and do not contain cliques of size three. Therefore, for every $p\ge 3$, it is NP-hard to compute the independence gap in the class of $K_p$-free perfect graphs. However, as we show next, in every such graph class there is a polynomial-time algorithm to recognize graphs of constant independence gap. The result will rely on a characterization of independence gap of a semi-perfect graph. A graph $G$ is said to be \emph{semi-perfect} if $\alpha(G)=\theta(G)$, where $\theta(G)$ denotes the clique cover number of $G$, that is, the minimum number of cliques covering $V(G)$. In other words, a graph $G$ is semi-perfect if and only if $\omega(\overline{G})=\chi(\overline{G})$, where $\omega(\overline{G})$ denotes the clique number of $\overline{G}$ and $\chi(\overline{G})$ its chromatic number. Recall that a graph $G$ is \emph{perfect} if $\chi(G) = \omega(H)$ for every induced subgraph $H$ of $G$. Since the complement of a perfect graph is perfect~\cite{MR0302480}, every perfect graph is semi-perfect. A \emph{clique partition} of a graph $G$ is a set of pairwise disjoint cliques with union $V(G)$. Thus, a graph is semi-perfect if and only if it has a clique partition of size $\alpha(G)$. We will refer to such a clique partition as an \emph{$\alpha$-clique partition} of $G$.

A clique is called \emph{strong} if it intersects every maximal independent set. A graph admitting a partition of its vertex set into strong cliques is called \emph{localizable}; these graphs are studied in \cite{localizable} (see also~\cite{HMR2018}). Here, we extend the concept defining localizable graphs as follows.

\begin{definition}
Let $k$ be a positive integer. A clique partition of $G$ is said to be \emph{$k$-tight} if the union of every $k$ cliques in the partition intersects all maximal independent sets. Furthermore, a clique partition is \emph{tight} if it is $1$-tight.
\end{definition}

Notice that a tight clique partition is exactly a partition of $V(G)$ into strong cliques; thus, localizable graphs are exactly the graphs admitting a tight clique partition. In Theorem 2.1 of \cite{localizable}, it has been shown that localizable graphs are exactly the graphs that are semi-perfect and well-covered. Several other characterizations of localizable graphs were given in that theorem. Using the concepts of independence gap and tight clique partitions, the theorem can be stated as follows.

\begin{theorem}[Hujdurovi\'c et al.~\cite{localizable}]\label{thm:alpha-chi-bar}
For every graph $G$, the following statements are equivalent.
\begin{enumerate}
  \item[(a)] $G$ has a tight clique partition.
  \item[(b)] $G$ has a tight $\alpha$-clique partition.
  \item[(c)] $G$ has an $\alpha$-clique partition and every $\alpha$-clique partition is tight.
  \item[(d)] $G$ is a semi-perfect graph with $\mu_\alpha(G) = 0$.
  \item[(e)] $i(G) = \theta(G)$.
\end{enumerate}
\end{theorem}

In particular, in view of the equivalence between property $(d)$ and the remaining properties, Theorem~\ref{thm:alpha-chi-bar} characterizes which semi-perfect graphs are well-covered. We now generalize this result to a characterization of semi-perfect graphs with any upper bound on the independence gap.

\begin{theorem}\label{thm:semiperfect}
For every semi-perfect graph $G$ and positive integer $k$, the following statements are equivalent.
\begin{enumerate}
  \item[(a)] $G$ has a $k$-tight clique partition.
  \item[(b)] $G$ has a $k$-tight $\alpha$-clique partition.
  \item[(c)] Every $\alpha$-clique partition of $G$ is $k$-tight.
  \item[(d)] $\mu_\alpha(G) \le k-1$.
\end{enumerate}
\end{theorem}

\begin{proof}
The implication $(b)\Rightarrow (a)$ is trivial. Since $G$ is semi-perfect, it has an $\alpha$-clique partition, which establishes the implication $(c)\Rightarrow (b)$.

$(a) \Rightarrow (d)$: We need to show that if $G$ has a $k$-tight clique partition, then $\mu_\alpha(G)\leq k-1$. Let $\{C_1,\ldots,C_{\ell}\}$ be a $k$-tight clique partition of $G$. Then $\ell\geq \theta(G) = \alpha(G)$. Suppose for a contradiction that $\mu_\alpha(G)\geq k$. Then there exists a maximal independent set $I$ of $G$ with $|I|\leq \alpha(G)-k$. Let $J'=\{j\mid I\cap C_j=\emptyset\}$. Since $I$ can contain at most one vertex from each clique in $C_1,\ldots,C_{\ell}$, the definition of $J'$ implies that $|J'| = \ell-|I|\geq \alpha(G)-|I| \geq k$. It is now clear that for every $J\subseteq J'$ of size $k$, we have $I\cap (\cup_{j\in J}C_j)=\emptyset$. This contradicts the assumption that $\{C_1,\ldots,C_{\ell}\}$ is $k$-tight. It follows that $\mu_\alpha(G)\leq k-1$.

$(d) \Rightarrow (c)$: Let $G$ be a semi-perfect graph with $\mu_\alpha(G)\leq k-1$. Since $G$ is semi-perfect, $G$ has an $\alpha$-clique partition. Consider an $\alpha$-clique partition $\{C_1,\ldots,C_{\alpha(G)}\}$ of $G$ and assume for a contradiction that it is not $k$-tight. Then there exists a set $J \subseteq \{1,\ldots,\alpha(G)\}$ with $|J|=k$ and a maximal independent set $I$ such that $I \cap C_j=\emptyset$  for every $j\in J$. Hence, $|I|\leq \alpha(G)-|J|=\alpha(G)-k$. It follows that $\alpha(G)-|I|\geq k$, contradicting the assumption that $\mu_\alpha(G)\leq k-1$.
\end{proof}

\begin{corollary}\label{cor:semiperfect-value-of-the-gap}
For every semi-perfect graph $G$, we have
$$\mu_\alpha(G) = \min\{k\ge 1\mid G\textrm{ has a $k$-tight clique cover}\}-1\,.$$
\end{corollary}

Theorem~\ref{thm:semiperfect} has the following algorithmic consequence.

\begin{theorem}\label{thm:semi-perfect}
For every pair of constants $p$ and $k$, given a $K_p$-free semi-perfect graph $G$ equipped with an optimal clique partition,
it can be determined in polynomial time whether $\mu_\alpha(G)\le k$.
\end{theorem}

\begin{proof}
Let $\mathcal{K}=\{C_1,\ldots,C_{\alpha(G)}$\} be an optimal clique partition of $G$. By Theorem~\ref{thm:semiperfect}, $G$ has independence gap at most $k$ if and only if $\mathcal{K}$ is $(k+1)$-tight. Testing whether $\mathcal{K}$ is $(k+1)$-tight can be done in polynomial time. Indeed, $\mathcal{K}$ is not $(k+1)$-tight if and only if for some subset $\mathcal{K'}$ of $(k+1)$ cliques from $\mathcal{K}$, there is an independent set that is disjoint from the union of cliques in $\mathcal{K'}$ that dominates the set of vertices in this union. Now, if such an independent set $I$ exists, then there is also an independent set $I'\subseteq I$ that is minimal with respect to the property of dominating $\mathcal{K'}$. By minimality, every vertex of $I'$ dominates at least one vertex in the union of cliques in $\mathcal{K'}$ that is not dominated by any other vertex of $I'$; thus, we conclude that $|I'|\leq (k+1)\omega(G)\le (k+1)(p-1)$.
Consequently, it is enough to check for every independent set $S$ of $G$ of size at most $(k+1)(p-1)$ whether $S$ dominates vertices in at least $k+1$ cliques from $\mathcal{K}$ that are disjoint from $S$.
Since $p$ and $k$ are constant, this can be done in polynomial time. \end{proof}

Using Theorem~\ref{thm:semi-perfect} we can now easily derive the following result.

\begin{corollary}\label{cor:perfect}
For every pair of constants $p$ and $k$, given any $K_p$-free perfect graph $G$, it can be decided in polynomial time whether $\mu_{\alpha}(G)\le k$.
\end{corollary}

\begin{proof}
Let $G$ be a $K_p$-free perfect graph. An optimal clique partition of $G$ can be computed in polynomial time~\cite{grotschel1981ellipsoid}. Since every perfect graph is semi-perfect, the result follows from Theorem \ref{thm:semi-perfect}.
\end{proof}

Corollary~\ref{cor:perfect} implies that for every $k\ge 0$, there is a polynomial-time algorithm for recognizing bipartite graphs with independence gap at most $k$. It also generalizes the fact that for every positive integer $p$, there is a polynomial-time algorithm for checking whether a given $K_p$-free perfect graph is well-covered. This was proved by Dean and Zito~\cite{MR1264476}.

\medskip
\begin{remark}
Unless P = NP, the assumptions that $p$ and $k$ are constant are necessary in Corollary~\ref{cor:perfect}. Indeed, as noted in Corollary~\ref{cor:w-c}, for every integer $k\ge 0$, it is co-NP-complete to determine if a given weakly chordal graph $G$ satisfies $\mu_{\alpha}(G)\le k$. Since weakly chordal graphs are perfect~\cite{MR815392}, the same conclusion holds for perfect graphs. Furthermore, as already observed above, the NP-hardness of computing the independence gap in the class of bipartite graphs implies that for every integer $p\ge 3$, it is co-NP-complete to determine whether a given $K_p$-free perfect graph $G$ and integer $k$ satisfy $\mu_\alpha(G)\le k$.
On the other hand, we do not know whether the assumption of perfection is necessary. For every $k\ge 0$ and $p \ge 3$, the complexity of recognizing graphs of independence gap at most $k$ in the class of $K_p$-free graphs is open. In particular, for $k = 0$, the problem becomes that of testing if a given triangle-free graph is well-covered, the complexity of which already seems to be open (cf.~\cite{MR1254158}).
\end{remark}

\section{{\sc Independent Set} in graphs of small independence gap}\label{sec:computing-alpha}

We now turn our attention to the complexity of {\sc Independent Set} when restricted to graphs with constant independence gap. This is an interesting problem in view of the fact that in any such class of graphs a constant additive approximation to a maximum independent set can be obtained in linear time by a simple greedy algorithm.
Graphs of independence gap zero are exactly the well-covered graphs and for such graphs any maximal independent set is also maximum. It follows that {\sc Independent Set} is solvable in linear time in the class of graphs of independence gap (at most)~$0$. In contrast with this fact,
we show that restricting the sizes of maximal independent sets to at most two consecutive values does not render the problem any easier than in general graphs; that is, {\sc Independent Set} remains NP-complete in the class of graphs of independence gap at most one (and consequently, in the class of graphs of independence gap at most $k$, for every positive integer~$k$).

In the proof of the following theorem, we reformulate a reduction from \cite{robust03} which was used to show that determining whether a well-covered graph contains an independent set of size $k$ cannot be solved in polynomial time by a so-called ``robust'' algorithm unless P = NP. The same reduction allows us to show that the recognition of well-covered graphs remains co-NP-complete when restricted to graphs having independence gap at most one.

\begin{theorem}\label{thm:hardness}
The following statements hold:
\begin{enumerate}
\item [i)] {\sc Independent Set} is NP-complete in the class of graphs with independence gap at most $1$.
\item [ii)] The problem of recognizing well-covered graphs is co-NP-complete in the class of graphs with independence gap at most $1$.
\end{enumerate}
\end{theorem}

\begin{proof}
For both results, we use a reduction from {\sc Independent Set} in general graphs, which is an NP-complete problem~\cite{MR519066}. Given a graph $G$ and an integer $k\ge 2$, we construct a graph $G'$ with $\mu_{\alpha}(G')\le 1$ as follows. For each vertex $v_i \in V(G)$, we create $k$ copies $v_{i,1}, \dots, v_{i,k}$ of $v_i$, referred to as \emph{$v$-type vertices}, in $V(G')$. We additionally create $k(k-1)$ vertices $u_{i,j}$, where $1 \le i \neq j \le k$, referred to as \emph{$u$-type vertices}. This completes the description of $V(G')$.
The first and second indices of a vertex of $G'$ are called its \emph{row number} and \emph{column number}, respectively. In other words, vertices with the same first (respectively, second) indices are said to belong to the same row (respectively, column).

We describe $E(G')$ in the sequel. There are five types of edges:
column edges, row edges, diagonal edges, cross edges, and $G$-edges.
\emph{Column edges} join all pairs of vertices (either of the same or of different types) in the same column, while \emph{row edges} join all pairs of $v$-type vertices in the same row. \emph{Diagonal edges} join all pairs of $u$-type vertices in different rows and different columns. \emph{Cross edges} join a $u$-type vertex and a $v$-type vertex if and only if the first index of the $u$-type vertex equals to the second index of the $v$-type vertex. Finally, \emph{$G$-edges} exist between $v$-type vertices if and only if the corresponding vertices are adjacent in $G$; that is, $v_{ij}v_{i'j'}\in E(G')$ if and only if $v_{i}v_{i'} \in E(G)$.

\begin{claim}\label{cl:wc}
Let $G'$ be the graph obtained from $(G,k)$ using the above transformation. Then the following statements hold.\\
a) $i(G') = k-1$.\\
b) $G'$ has independence gap at most one.\\
c) $\alpha(G)\ge k$ if and only if $\alpha(G')\ge k$.\\
d) $\alpha(G)\ge k$ if and only if $G'$ is not well-covered.
\end{claim}

\begin{sloppypar}
\begin{proof}[Proof of the claim]
$a)$ We first show that every maximal independent set $S$ in $G'$ has size at least $k-1$. Suppose that there is a maximal independent set $S$ with size at most $k-2$. Then there are at least two different columns $i$ and $j$ that do not contain any vertex from $S$. Consider the $u$-type vertex $u_{i,j}$. Since $u_{i,j} \notin S$, the maximality of $S$ and the fact that the only $v$-type neighbors of $u_{i,j}$ are in columns $i$ and $j$ imply that a $u$-type neighbor $u_{i',j'}$ of $u_{i,j}$ must be in $S$. Due to the diagonal edges and since there are no vertices in $S$ from columns $i$ and $j$, we have $i' \neq i$ and $j' \neq j$. Hence, $S$ does not contain a $v$-type vertex from column $i'$ and the only $u$-type vertices in $S$ are in row $i'$ (due to diagonal edges). Then vertex $u_{i',i}$ can be added to $S$, a contradiction. Therefore, every maximal independent set $S$ has size at least $k-1$.
This shows that $i(G')\ge k-1$.

To establish the converse inequality, we show that every maximal independent set containing a $u$-type vertex has size at most $k-1$. Let $S$ be a maximal independent set in $G'$ containing a $u$-type vertex $u_{i,j}$. Due to the column edges, $S$ cannot contain any other $u$-type vertex from column $j$. Furthermore, $S$ contains no $v$-type vertex from column $i$ because of the cross edges from $u_{i,j}$ and no $u$-type vertex from any row other than $i$ due to the diagonal edges from $u_{i,j}$. Note that there are exactly $k-1$ $u$-type vertices in each row. Hence, $S$ can contain at most $k-2$ other $u$-type vertices which are all in the same row as $u_{i,j}$ and every $u$-type vertex in $S$ other than $u_{i,j}$ excludes an additional column. Moreover, $S$ can have at most $k-2$ $v$-type vertices since columns $i$ and $j$ are excluded due to column edges. Combining these observations, we conclude that $S$ has size at most $k-1$. This shows $i(G')\le k-1$, hence equality holds.

\medskip
\noindent $b)$ At most one vertex from each column can belong to an independent set of $G'$. Therefore, $\alpha(G') \le k$ and, using also statement $a)$, we have
$\mu_\alpha(G')= \alpha(G')-i(G')\le 1$.

\medskip
\noindent $c)$ Suppose first that $G$ has an independent set of size $k$, say $S=\{v_{i_1}, \dots, v_{i_k}\}$. Then $\{v_{{i_1,1}}, v_{{i_2,2}}, \dots, v_{{i_k,k}}\}$ is an independent set of size $k$ in $G'$. Hence
$\alpha(G)\ge k$ implies that $\alpha(G')\ge k$.
Conversely, suppose that $\alpha(G)<k$, that is, $G$ has no independent set of size $k$. We already showed in part $a)$ that all maximal independent sets of $G'$ containing a $u$-type vertex have size at most $k-1$. If $S$ is a maximal independent set in $G'$ that does not contain any $u$-type vertex, then the row edges and $G$-edges ensure that $S$ corresponds to an independent set in $G$. Therefore, $S$ has size at most $k-1$ in this case as well. Hence all maximal independent sets in $G'$ have size $k-1$.

\medskip
\noindent $d)$ Since $i(G') = k-1$ by statement $a)$, graph $G'$ is not well-covered if and only $\alpha(G')\ge k$. Therefore, statement $d)$ follows from $c)$.
\end{proof}
\end{sloppypar}

Statement $c)$ of Claim~\ref{cl:wc} and the fact that {\sc Independent Set} is an NP-complete problem in general graphs imply part i) of the theorem. Furthermore, since the problem of recognizing well-covered graphs is in co-NP for general graphs, it is also in co-NP for graphs with independence gap at most $1$. Hence $ii)$ follows from statement $d)$ of Claim~\ref{cl:wc}.
\end{proof}

It should be noted that the above result does not settle the complexity status of {\sc Independent Set} for graphs having independence gap exactly 1, which we leave as an open question.

\section{A hereditary version of independence gap}\label{sec:hereditary}

We say that a graph invariant $\pi$ is \emph{monotone under induced subgraphs} (or simply: \emph{hereditary}) if $\pi(G_1)\le \pi(G_2)$ whenever $G_1$ is an induced subgraph of $G_2$. Note that independence gap is not hereditary: deleting a vertex from a graph may change the value of the independence gap in either direction.\footnote{For instance, the independence gaps of paths $P_4$, $P_3$, and $P_2$ are $0$, $1$, and $0$, respectively.} In fact, for every $k\ge 0$, the class $\mathcal{G}_k$ of graphs of independence gap at most $k$ is not hereditary. To see this, consider for example the graph $G$ which is the complete bipartite graph $K_{1,k+2}$ and let $G'$ be the graph obtained from $G$ by adding a private neighbor to every vertex of $G$. Then $G'$ is well-covered. (This can be seen, for example, by noticing that the edges incident with vertices of degree one form a tight clique partition and applying Theorem~\ref{thm:alpha-chi-bar}.) Hence $\mu_\alpha(G')\le k$, that is, $G'\in \mathcal{G}_k$. However, graph $G$, which is an induced subgraph of $G'$, has $i(G) = 1$ and $\alpha(G) = k+2$, implying that $\mu_\alpha(G) = k+1$ and hence $G\not\in \mathcal{G}_k$.

A natural way to turn the independence gap into a hereditary invariant is as follows. We denote the \emph{hereditary independence gap} of a graph $G$ by $\overline{\mu_{\alpha}}(G)$
and define it as $$\overline{\mu_{\alpha}}(G)= \max\{\mu_{\alpha}(H)\mid H\subseteq_i G\}$$ where $\subseteq_i$ denotes the induced subgraph relation. Note that a graph $G$ has hereditary independence gap at most $k$ if and only if every induced subgraph $H$ of $G$ has independence gap at most $k$. In particular, if $H$ is an induced subgraph of a graph $G$, then $\overline{\mu_\alpha}(H)\le \overline{\mu_\alpha}(G)$. This leads to a family of hereditary graph classes related to independence gap, one for each non-negative integer upper bound on the value of the hereditary independence gap.

Graphs with hereditary independence gap (at most) $0$ are easy to characterize. Note that the path $P_3$ is not well-covered, while every proper induced subgraph of $P_3$ is. Therefore, $P_3$ is a forbidden induced subgraph for the class of graphs of hereditary independence gap (at most) $0$. Furthermore, since every $P_3$-free graph is a disjoint union of complete graphs and all such graphs are well-covered, the graphs of hereditary independence gap at most $0$ are exactly the $P_3$-free graphs. In what follows, we extend this simple observation by showing that for every non-negative integer $k$, the class of graphs having hereditary independence gap at most $k$ is characterized by a finite set of forbidden induced subgraphs, all of which are bipartite.

\subsection*{\it The case $k = 1$.}

For graphs of hereditary independence gap at most one we are able to obtain a precise characterization, from which, as a side result, we will identify a new polynomially solvable case of {\sc Independent Domination}. The \emph{claw} is the complete bipartite graph $K_{1,3}$ and $2P_3$ denotes the graph consisting of two copies of the $3$-vertex path.

\begin{theorem}\label{thm:hereditary}
A graph $G$ has $\overline{\mu_{\alpha}}(G) \le 1$ if and only if $G$ is $\{$claw, $2P_3\}$-free.
\end{theorem}

\begin{proof}
Since the claw and the $2P_3$ have independence gap $2$, every graph $G$ with $\overline{\mu_{\alpha}}(G) \le 1$ is $\{$claw, $2P_3\}$-free.
For the other direction, it suffices to prove that for every $\{$claw, $2P_3\}$-free graph, it holds that $\mu_{\alpha}(G) \le 1$.
Suppose for a contradiction that $G$ is a $\{$claw, $2P_3\}$-free graph of independence gap at least~$2$ and let
$I_1$ and $I_2$ be two maximal independent sets in $G$ such that $|I_1|\ge |I_2|+2$.
Let $H$ be the subgraph of $G$ induced by the symmetric difference $I_1\triangle I_2$ (defined as
$(I_1\setminus I_2)\cup(I_2\setminus I_1)$). Graph $H$ is bipartite and, being claw-free, of maximum degree at most~$2$. Hence, $H$
is a disjoint union of paths and even cycles. Since $|I_1|\ge |I_2|+2$, at least two components of $H$ are
paths of even length. However, this implies that $H$, and hence $G$, contains an induced $2P_3$, a contradiction.
\end{proof}

Theorem~\ref{thm:hereditary}, together with known results from the literature on claw-free graphs, can be used to infer that {\sc Independent Domination} is solvable in polynomial time in the class of $\{$claw, $2P_3\}$-free graphs.

\begin{corollary}\label{cor:inddom}
{\sc Independent Domination} can be solved in polynomial time in the class of $\{$claw, $2P_3\}$-free graphs.
\end{corollary}

\begin{proof}
Given a $\{$claw, $2P_3\}$-free graph $G$, we compute the independent domination number $i(G)$ as follows.
First, we compute the independence number of $G$ using any of the polynomial-time algorithms for computing $\alpha(G)$ in claw-free graphs
(e.g., the one due to Minty~\cite{MR579076} or Sbihi~\cite{MR553650}). Then, we check whether $G$ is well-covered, using one of the polynomial-time recognition algorithms for claw-free well-covered graphs due to Tankus and Tarsi~\cite{MR1438624,MR1376052}.
If $G$ is well-covered, then $i(G) = \alpha(G)$ and we return this value.
If $G$ is not well-covered, then, by Theorem~\ref{thm:hereditary}, we have $\mu_\alpha(G) =1$, and so
$i(G) = \alpha(G)-1$ in this case.
\end{proof}

The result of Corollary~\ref{cor:inddom} is sharp with respect to both forbidden induced subgraphs. It is known that {\sc Independent Domination} is NP-complete in the class of claw-free graphs~\cite{MR1975229}. Moreover, for every $\epsilon>0$, the problem cannot be approximated to within a factor of $n^{1-\epsilon}$ in the class of $n$-vertex $2P_3$-free graphs, unless P = NP~\cite{MR2492039}. Other similar but incomparable polynomial-time solvable cases of {\sc Independent Domination} include the classes of $\{P_5, 2P_3\}$-free graphs and $\{P_5, K_{2,3}\}$-free graphs~\cite{MR3301930}, $\{\mbox{claw}, P_6\}$-free graphs \cite{MR1975229}, and $\{$claw, co-claw$\}$-free graphs where a co-claw is the complement of a claw; this latter class has bounded clique-width and admits a linear-time algorithm for computing a clique-width decomposition with a bounded number of labels~\cite{clawcoclaw}, which in turn allows us to solve the weighted version of {\sc Independent Domination} in linear time~\cite{MR1739644}.

\subsection*{\it The general case of fixed $k$.}

Let us now turn our attention to forbidden induced subgraphs for classes of graphs with hereditary independence gap at most $k$, where $k\ge 2$.
In a similar way to graphs with hereditary independence gap at most 1, it can be observed that $K_{1,k+2}$ and $(k+1)P_3$ are two forbidden induced subgraphs. However, for $k\ge 2$, they do not form the complete list; for example, the disjoint union of $K_{1,k+1}$ and $P_3$ does not contain $K_{1,k+2}$ or $(k+1)P_3$ as induced subgraphs but still has hereditary independence gap of $k+1$. Moreover, with suitable edge additions such examples can be easily extended to forbidden induced subgraphs that are connected (and bipartite). Nonetheless, the list of forbidden induced subgraphs remains finite for every $k$.

\begin{theorem}\label{thm:hereditary-k}
For every non-negative integer $k$, there is a finite set of bipartite graphs $\mathcal{F}_k$ such that a graph $G$ has $\overline{\mu_{\alpha}}(G) \le k$ if and only if it is $\mathcal{F}_k$-free.
\end{theorem}

\begin{proof}
Fix $k\ge 0$ and let $\mathcal{F}_k$ denote the set of minimal graphs of independence gap at least $k+1$, where minimal is in the sense of induced subgraphs, that is, $\mu_\alpha(G) \ge k+1$ but $\mu_\alpha(H)\le k$ for all proper induced subgraphs $H$ of $G$. Clearly, a graph $G$ has $\overline{\mu_{\alpha}}(G) \le k$ if and only if $G$ is $\mathcal{F}_k$-free.
First, we show that $\{K_{1,k+2}, (k+1)P_3\}\subseteq \mathcal{F}_k$, where
$K_{1,k+2}$ denotes the complete bipartite graph with parts of size $1$ and $k+2$, and
$(k+1)P_3$ denotes the disjoint union of $k+1$ copies of $P_3$.
Clearly, $\alpha(K_{1,k+2}) = k+2$ and $i(K_{1,k+2}) = 1$, which implies $\mu_\alpha(K_{1,k+2}) = k+1$, and
it is easy to see that every proper induced subgraph of $K_{1,k+2}$ is of independence gap at most $k$.
Similarly, $\alpha((k+1)P_3) = 2(k+1)$ and $i((k+1)P_3) = k+1$, which implies $\mu_\alpha((k+1)P_3) = k+1$, and
again it can be readily checked that every proper induced subgraph of $(k+1)P_3$ is of independence gap at most $k$.

Let $\mathcal{F}_k'$ denote the set of all bipartite graphs in $\mathcal{ F}_k\setminus\{K_{1,k+2}, (k+1)P_3\}$.
We will now show that $\mathcal{F}_k'$ is finite. Let $F\in \mathcal{F}_k'$. If $F$ has a component $C$ that is well-covered, then $F-V(C)$ would be a proper induced subgraph of $F$ with independence gap $k+1$, contrary to the minimality of $F$. In particular, no component of $F$ is complete, which implies that every component of $F$ has an induced $P_3$. Since every proper induced subgraph of $F$ is of independence gap at most $k$ and $F\neq (k+1)P_3$, we infer that $F$ is $(k+1)P_3$-free; in particular, $F$ has at most $k$ components. Similarly, the fact that $F\neq K_{1,k+2}$ implies that $F$ is $K_{1,k+2}$-free; hence, since $F$ is bipartite, $F$ is of maximum degree at most $k+1$. Moreover, since $F$ is $(k+1)P_3$-free, its diameter is at most $4k+1$.
The fact that the maximum degree and diameter are bounded implies that there is a constant $c_k$ such that
every component of $F$ has at most $c_k$ vertices. Consequently, $F$ has at most $kc_k$ vertices and the set $\mathcal{F}_k'$ is finite.

To complete the proof, it suffices to show that
$\mathcal{F}_k = \mathcal{F}_k'\cup\{K_{1,k+2}, (k+1)P_3\}$, that is, all minimal graphs in $\mathcal{F}$ are bipartite. The inclusion $\mathcal{ F}_k'\cup\{K_{1,k+2}, (k+1)P_3\}\subseteq \mathcal{F}_k$ is clear.
Suppose that the converse inclusion is false, that is, there exists a graph $F\in \mathcal{F}_k\setminus (\mathcal{F}_k'\cup\{K_{1,k+2}, (k+1)P_3\})$.
Since $F \not\in (\mathcal{F}_k'\cup\{K_{1,k+2}, (k+1)P_3\})$, the fact that $F\in \mathcal{F}_k$ implies that
$F$ is $\mathcal{F}_k'\cup\{K_{1,k+2}, (k+1)P_3\}$-free.
Since $F$ is of independence gap at least $k+1$, there are two maximal independent sets, say $I_1$ and $I_2$, in $F$ such that
$|I_1|\ge |I_2|+k+1$. Adopting a similar approach as in the proof of Theorem~\ref{thm:hereditary}, let
$H$ be the subgraph of $F$ induced by the symmetric difference $I_1\triangle I_2$.
Graph $H$ is bipartite and $(\mathcal{F}_k'\cup\{K_{1,k+2}, (k+1)P_3\}$)-free.
Since the set $\mathcal{F}_k'\cup\{K_{1,k+2}, (k+1)P_3\}$ is exactly the set of bipartite graphs in $\mathcal{F}_k$,
we infer that $H$ is of independence gap at most $k$.
However, the fact that $I_1$ and $I_2$ are maximal independent sets in $F$ implies that every vertex in $I_1\setminus I_2$
has a neighbor in $I_2\setminus I_1$ and vice versa. Thus, $I_1\setminus I_2$ and $I_2\setminus I_1$ are maximal independent sets in $H$ differing in size by at least $k+1$, contrary to the fact that $\mu_\alpha(H)\le k$.
This completes the proof.
\end{proof}

Theorem~\ref{thm:hereditary-k} has the following algorithmic consequence.

\begin{corollary}\label{cor:hereditarily-k-quasi-well-covered}
For every non-negative integer $k$, there is a polynomial-time algorithm for recognizing graphs with hereditary independence gap at most $k$.
\end{corollary}

Recall that we established in Theorem~\ref{thm:hardness} that {\sc Independent Set} is
co-NP-complete in the class of graphs of (non-hereditary) independence gap at most 1. In contrast with this fact, another consequence of Theorem \ref{thm:hereditary-k}, when combined with known results in the literature, is a polynomial-time algorithm for the weighted version of the {\sc Independent Set} problem in the class of graphs with hereditary independence gap at most $k$.

\begin{corollary}\label{cor:WIS}
For every non-negative integer $k$, {\sc Weighted Independent Set}
is polynomial-time solvable in the class of graphs of hereditary
independence gap at most $k$.
\end{corollary}

\begin{proof}
Let $G$ be a graph with $\overline{\mu_{\alpha}}(G)\le k$. By the proof of Theorem \ref{thm:hereditary-k}, $G$ is $\{K_{1,k+2}, (k+1)P_3\}$-free; thus, it is also $P_{4k+3}$-free. We conclude by applying a result of
Lozin and Rautenbach~\cite{WIS} stating that for every two positive integers $\ell$ and $\ell'$, {\sc Weighted Independent Set} is polynomial-time solvable in the class of $\{P_{\ell}, K_{1,\ell'}\}$-free graphs.
\end{proof}

\subsection*{\it The case when $k$ is part of input.}

We conclude this section by observing that, in contrast to Corollary \ref{cor:hereditarily-k-quasi-well-covered}, a simple reduction from  {\sc Independent Set} shows that the problem of computing the hereditary independence gap is NP-hard.

\begin{theorem}\label{thm-np-hardness}
Given a graph $G$ and an integer $k$, it is co-NP-complete to determine whether $\overline{\mu_{\alpha}}(G) \le k$.
\end{theorem}

\begin{proof}
If $(G,k)$ is a no instance to the problem of determining whether $G$ has $\overline{\mu_{\alpha}}(G) \le k$,
then there exists an induced subgraph $H$ of $G$ having two maximal independent sets $I_1$ and $I_2$ such that
$|I_1|\ge |I_2|+k+1$. Therefore, the triple $(H,I_1,I_2)$ forms a polynomially verifiable certificate of the fact that $(G,k)$ is a no instance. This shows that the problem is in co-NP.

We prove hardness by a reduction from the NP-hard {\sc Independent Set} problem. Let $(G,k)$ be an instance to the independent set problem; we may assume that $k\ge 2$.
Let $G'$ be the graph obtained from $G$ by adding to it a new vertex $v$ and making it adjacent to all vertices of $G$.
To complete the proof, we claim that $\alpha(G)\ge k$ if and only if $G'$ has $\overline{\mu_{\alpha}}(G) \ge k-1$.
Suppose first that $\alpha(G)\ge k$ and let $I$ be an independent set in $G$ of size $k$.
Then, the subgraph of $G'$ induced by $I\cup \{v\}$ is isomorphic to the complete bipartite graph $K_{1,k}$ and hence of independence gap $k-1$.
This implies that $G'$ has $\overline{\mu_{\alpha}}(G) \ge k-1$.
Conversely, suppose that $G'$ has $\overline{\mu_{\alpha}}(G) \ge k-1$.
Then $G'$ has an induced subgraph $H$ such that $\mu_\alpha(H)\ge k-1$.
Let $I_1$ and $I_2$ be two maximal independent sets in $H$ such that $|I_1|\ge |I_2|+k-1$.
Since $|I_1|\ge k\ge 2$ and $v$ is adjacent to all vertices of $G$, we have $v\not\in I_1$.
Thus, $I_1$ is an independent set in $G$ of size at least $k$.
\end{proof}

\section{Summary and open questions}\label{sec:sum}

In Table \ref{tab:sum}, we summarize some known results about independence gap and hereditary independence gap for various graph classes. All results with a reference are known from the literature, whereas results obtained in this paper are shown in bold with a reference to the related statement therein. All results without a reference follow from one of the referenced results as a consequence of the containment of one graph class in another one (if a result is implied by another result in the same column), or as a consequence of the fact that the problem under consideration is at least as difficult as another problem (if a result is implied by another result in the same row). Moreover, if both {\sc Independent Domination} (column 4) and {\sc Independent Set} (column 5) can be solved in polynomial time, then one can also compute the independence gap of the given graph in polynomial time.
Cells containing a ``?'' sign designate problems whose complexity in the corresponding graph class is unknown. In particular, we leave a more detailed investigation of algorithmic aspects of hereditary independence gap
as an interesting problem for future work.

\begin{table}
\begin{center}
\footnotesize
  \begin{tabular}{|c|c|c|c|c|c|c|c|}
\hline   & $\mu_{\alpha} \le k$ & $\mu_{\alpha} \le k$ & $\mu_{\alpha} = 0$ & $i \le k$ & $\alpha \ge k$ & $\overline{\mu_{\alpha}}\le k$ &  $\overline{\mu_{\alpha}}\le k$\\
     & $k$ input & $k$ const. & (well-covered) & $k$ input & $k$ input & $k$ input
      & $k$ const. \\
          \hline
    Perfect, & & & & & & &\\
    $\omega$ const. & co-NP-c & \bf{P (\ref{cor:perfect})} & P~\cite{MR1264476} & NP-c & P \cite{grotschel1981ellipsoid} & ? & P\\ \hline
    Bipartite & co-NP-c & P & P~\cite{MR0469831} & NP-c \cite{zvervich1995induced} & P & ? & P\\ \hline
    Chordal & P & P & P~\cite{MR1368737} & P \cite{MR687354}& P \cite{gavril1972algorithms}& ? & P\\ \hline
   Weakly Chordal & co-NP-c & co-NP-c & co-NP-c~\cite{MR1161178,MR1217991} & NP-c~\cite{MR2207507} & P & ? & P\\ \hline
    $\{\mbox{claw},2P_3\}$-free & P & P & P & \bf{P (\ref{cor:inddom})} & P ~\cite{MR579076, MR553650} & P~(see~(\ref{thm:hereditary})) & P \\ \hline
    All & co-NP-c & {co-NP-c}~\cite{MR1254158} & co-NP-c & NP-c & NP-c  & \bf{co-NP-c} & \bf{P}\\
     & & & & & & \bf{(\ref{thm-np-hardness})} &  \bf{(\ref{cor:hereditarily-k-quasi-well-covered})}\\ \hline
  \end{tabular}
\end{center}
\caption{Summary of complexity results.}\label{tab:sum}
\end{table}

We also obtained some new complexity results for {\sc Independent Set}. The problem is NP-complete in the class of graphs with independence gap at most $1$, whereas its weighted version can be solved in polynomial time for graphs of bounded hereditary independence gap. Furthermore, we have also showed that the complexity of recognizing well-covered graphs is co-NP-complete even when the graph is restricted to have independence gap at most one. As noted at the end of Section \ref{sec:indgap}, the complexity of {\sc Independent Set} for graphs having independence gap exactly~1 does not follow from the results in this paper; we leave this as an open question. Finally, our work leaves open the complexity status of recognizing (non-perfect) graphs of bounded clique number and bounded independence gap.

\subsection*{Acknowledgements}

This research has been carried out when T\i naz Ekim was spending her academic leave at the University of Oregon, Department of Computer and Information Science supported  by Fulbright Scholarship and The Scientific and Technological Research Council of Turkey TÜB\.ITAK 2219 Grant. She is also supported by the Turkish Academy of Sciences GEBIP award.
This work is supported in part by the Slovenian Research Agency (I0-0035, research program P1-0285 and research projects N1-0032, N1-0038, J1-7051, N1-0062, and J1-9110).

\end{document}